\documentclass[oneside,reqno]{amsart}
\usepackage{amssymb,amsmath}
\usepackage[english]{babel}
\usepackage{amsfonts}
\usepackage{color}
\usepackage{amsthm}
\usepackage[colorlinks=true,linkcolor=NavyBlue, citecolor=NavyBlue,urlcolor=NavyBlue]{hyperref}
\usepackage{graphicx}
\usepackage{mathtools}
\usepackage[usenames,dvipsnames]{pstricks}
\usepackage{bm}
\usepackage{amsmath}
\usepackage{listings}
\usepackage{hyperref}

\usepackage{xy}
\usepackage[draft]{fixme}
\DeclareMathOperator{\AGL}{AGL}

\DeclareMathOperator{\Sym}{Sym}

\DeclareMathOperator{\Stab}{Stab}

\newcommand\deq{\mathrel{\stackrel{\makebox[0pt]{\mbox{\normalfont\tiny def}}}{=}}}

\newcommand{\Size}[1]{\left\lvert #1 \right\rvert}
\newcommand{\Span}[1]{\left\langle#1\right\rangle}
\newcommand{\Set}[1]{\left\lbrace #1 \right\rbrace}

\let\phi\varphi
\theoremstyle{plain}
\newtheorem{theorem}{Theorem}[section]

\newtheorem{lemma}[theorem]{Lemma}
\newtheorem{proposition}[theorem]{Proposition}

\newtheorem{conjecture}{Conjecture}
\newtheorem{openprob}{Open Problem}

\theoremstyle{remark}
\newtheorem{remark}{Remark}

\theoremstyle{definition}
\newtheorem{definition}[theorem]{Definition}

\newtheorem*{notation*}{Notation}

\usepackage{listings}
\lstdefinelanguage{GAP}{%
 morekeywords={%
 Assert,Info,IsBound,QUIT,%
 TryNextMethod,Unbind,and,break,%
 continue,do,elif,%
 else,end,false,fi,for,%
 function,if,in,local,%
 mod,not,od,or,%
 quit,rec,repeat,return,%
 then,true,until,while%
 },%
 sensitive,%
 morecomment=[l]\#,%
 morestring=[b]",%
 morestring=[b]',%
}[keywords,comments,strings]

\usepackage[T1]{fontenc}
\usepackage[variablett]{lmodern}
\usepackage{xcolor}
\begin{document}
\title[A Chain of Normalizers in the Sylow $2$-subgroups of
SYM$(2^n)$]{A Chain of Normalizers in the Sylow $2$-subgroups of the
 symmetric group on $2^n$ letters} \author[R.~Aragona]{Riccardo
 Aragona}

\author[R.~Civino]{Roberto Civino}

\author[N.~Gavioli]{Norberto Gavioli}

\author[C.~M.~Scoppola]{Carlo Maria Scoppola}
\address{DISIM \\
 Universit\`a degli Studi dell'Aquila\\
 via Vetoio\\
 I-67100 Coppito (AQ)\\
 Italy}       \email{riccardo.aragona@univaq.it}
\email{roberto.civino@univaq.it} \email{norberto.gavioli@univaq.it}
\email{carlo.scoppola@univaq.it}

\date{} \thanks{All the authors are members of INdAM-GNSAGA
 (Italy). R. Civino is partially funded by the Centre of excellence
 ExEMERGE at University of L'Aquila. Part of this work has been
 carried out during the cycle of seminars ``Gruppi al Centro''
 organized at INdAM in Rome.}


\begin{abstract}
 On the basis of an initial interest in symmetric cryptography, in
 the present work we study a chain of subgroups. Starting from a
 Sylow $2$-subgroup of $\AGL(2,n)$, each term of the chain is defined
 as the normalizer of the previous one in the symmetric group on
 $2^n$ letters. Partial results and computational experiments lead us
 to conjecture that, for large values of $n$, the index of a
 normalizer in the consecutive one does not depend on $n$. Indeed,
 there is a strong evidence that the sequence of the logarithms of
 such indices is the one of the partial sums of the numbers of
 partitions into at least two distinct parts.
\end{abstract}

\maketitle


\section{Introduction}
Let $n$ be a non-negative integer and let $\Sym(2^n)$ denote the
symmetric group on $2^n$ letters. The study of the conjugacy class in
$\Sym(2^n)$ of the elementary abelian regular $2$-subgroups has
recently drawn attention for its application to block cipher
cryptanalysis, and in particular to differential cryptanalysis
\cite{bih91}. The reader which is familiar with symmetric cryptography
will not find hard to realize that the key-addition layer of a block
cipher (see e.g.~\cite{aes,present,des}) acts in general on the partially
encrypted states as an elementary abelian regular $2$-subgroup of the
message space. In a recent paper \cite{calderini2017}, it has been
shown that a cryptanalyst can derive from such subgroups new
operations on the message space $\mathbb F_2^n$ of the block cipher,
which can be used to perform algebraic and statistical
attacks. Indeed, although the encryption functions, in order to be
secure, are designed to be far from being linear with respect to the
classical bitwise addition modulo 2, it is possible to attack the
encryption scheme by means of a variation of the classical
differential attack, where instead a newly designed operation is used
\cite{Civino2018}. Such operation is defined starting from a conjugate
of the translation group $T$ on the message space. 

A study of regular
subgroups of the affine group is  carried out in \cite{cds06,
 catino2009} by means of radical algebras.
We point out that there is   an interesting connection between our study of the \emph{position} of a regular subgroup in the symmetric group, in terms of the chain of normalizers defined below, and the rather recent theory of \emph{braces}, introduced in~\cite{rump2007braces}, since the above mentioned  new operation  can  be used to construct a brace on $T$. 
	Indeed, when $+$ and $\circ$ respectively denote the (additive) laws induced by
	$T$ and by one of its affine conjugates, the structure $( T, +, \circ)$ is 
	a two-sided brace and $( T, +, \cdot)$ is a radical ring, where $a\cdot b$ is defined 
	as $a+b+a\circ b$ for each $a, b \in T$. For an extensive survey and  detailed references  on braces  see e.g.~\cite{Cedo2018}.  
	\\
	

%

In a recent paper~\cite{Aragona2019}, we considered the elements
of the conjugacy class $T^{\Sym(2^n)}$ which are subgroups of the affine group
$\AGL(T)$. We showed that, if ${T^g \cap T }$ has index 4 in $T$, then
there exists a Sylow $2$-subgroup $U < \AGL(T)$ containing both $T^g$
and $T$ as normal subgroups. The normalizer $N^1$ of $U$ in
$\Sym(2^n)$ contains $U$ as a subgroup of index 2 and interchanges $T$
and $T^g$ by conjugation. The $2$-group $N^1$ is therefore contained
in a Sylow $2$-subgroup $\Sigma$ of $\Sym(2^n)$. Motivated by a
computational evidence, we prove here that this is the general
behavior. We define a chain starting from $U$ and where the $k$-th
term $N^k$ is the normalizer in $\Sym(2^n)$ of the previous
$N^{k-1}$. We show that $N^k$ is actually the normalizer of $N^{k-1}$
in $\Sigma$, and thus the $N^k$s form a sequence of $2$-groups ending
at $\Sigma$. Philip Hall, indeed, proved that $\Sigma$ is
self-normalizing~(see e.g.\ \cite{Carter1964}). Using the software
package \textsf{GAP}~\cite{GAP4}, we computed the normalizer chain for
$n \leq 11$.  We experimentally
noticed that the sequence defined by $c_k = \log_2| N^k : N^{k-1}|$
does not depend on $n$ if $k\leq n-2$ and, in such cases, $\{c_k\}_{k \geq 1}$
represents the sequence of partial sum of the sequence $\{b_{k+2}\}_{k \geq 1}$, where $b_k$ counts the
number of partitions of $k$ into at least two distinct parts, a
well-known    sequence    of    integers
\cite[\url{https://oeis.org/A317910}]{OEIS}, also appearing in
commutative algebra problems \cite{Enkosky2014}. 
For larger values of $n$, the computational problem
becomes intractable using the standard libraries, and so its
investigation requires a theoretical approach.
For small values of $k$, by way of an elementary but increasingly cumbersome analysis,  we show that the previous claim is true. 
 In the
general case, the claim remains an open problem.
We believe that  more sophisticated combinatorial and group theoretical tools  could prove  that, for $k\le n-2$, the integers   $c_k$ do not depend on $n$ 
and are related to the sequence $b_k$ as previously mentioned.
\\

The paper is organized as follows: in Sec.~\ref{sec:one} we introduce
the notation and provide some preliminary results. The normalizer
chain is defined in Sec.~\ref{sec_two}, which contains the main
considerations that led us to formulate Conjecture~\ref{conj:main}. Some theoretical evidence in support of our
conjecture, i.e.\ Theorem~\ref{thm:main},  is proved in
Sec.~\ref{sec:evi}, where we also discuss some  open problems. To
conclude, Sec.~\ref{sec_construction} is devoted to the computational
aspects and contains the \textsf{GAP} code used for our computations.

\section{Notation and preliminaries}\label{sec:one}
In this section, we recall some well known facts and a preliminary
result on the imprimitivity action of subgroups of the symmetric group
on a finite set.
\begin{definition}
 Let $\Omega\ne \emptyset$ and let $G\le \Sym(\Omega)$ be a
 transitive permutation group. An imprimitivity system $\mathcal B$
 for $G$ is a $G$-invariant partition of $\Omega$. The group $G$ is
 primitive if $G$ has only the trivial partitions $\{\Omega\}$ and
 the set of the singletons of $\Omega$ as imprimitivity systems.
 Otherwise, $G$ is said to be imprimitive.
\end{definition}
 
 \begin{definition}
 Let $G$ act imprimitively on the set $\Omega$. An imprimitivity
 chain $\mathcal{B}_0 \succ \cdots \succ \mathcal{B}_t$ of depht $t$
 is a sequence of imprimitivity systems for $G$ acting on $\Omega$,
 where $\mathcal{B}_0$ and $\mathcal{B}_t$ are the trivial
 partitions. We also require that for each $B\in \mathcal{B}_{m+1}$
 there exists $ B' \in \mathcal{B}_{m}$ such that $B\subset {B'}$
 for $0\leq m \leq t-1$.
 \end{definition}
 Note   that   the   imprimitivity   chain
 $\mathcal{B}_0 \succ \cdots \succ \mathcal{B}_t$ can be represented
 by its \emph{imprimitivity tree} which is the rooted tree $(V,E)$,
 where
 \begin{itemize}
 \item the set of vertices $V$ is $\bigcup_{m=0}^t\mathcal{B}_m$, more precisely a vertex is a subset of $\Omega$ belonging to some partition $\mathcal{B}_i$ 
 and
 the root vertex is $\Omega$;
 \item two vertices $X$ and $Y$ in $V$ are connected by an edge
 $e\in E$ if and only if there exists $m$ such that
 $X\in \mathcal{B}_{m}$, $Y\in \mathcal{B}_{m+1}$ and $Y\subset X$.
 \end{itemize}

 In the remainder of this work, we will consider the special case of a
 subgroup $G$ of the symmetric group $\Sym(X_n)$, where
 $X_n\deq \{1,\ldots,2^n\}$.\\
 
 For $0 \leq m\leq n$ and $0 \leq k \leq 2^m-1$, the following
 notation is used:
 \begin{itemize}
 \item $B^n_{m,k}\deq \Set{k2^{n-m}+1,\ldots ,(k+1)2^{n-m}}$, and in
 particular $X_n=B^n_{0,0}$;

 \item $\mathcal{B}^n_m\deq \Set{ B^n_{m,0},\ldots , B^n_{m,2^m-1}} $;

 \item for $1 \leq i \leq n$
 \[s_i\deq \prod_{j=1}^{2^{i-1}}(j,j+2^{i-1});\]
 
 \item $t^n_i\deq \begin{cases}
  s_i & \text{if $i=n$}\\
  t_{i}^{n-1}\cdot (t^{n-1}_{i})^{s_n} & \text{if $1\le i < n$}.
 \end{cases}$
 
 \end{itemize}

 The symmetric group $\Sym(2^n)$ acts on the set of partitions of
 $X_n$ and, with respect to this action, we define the subgroup
 \[\Sigma_n\deq \bigcap_{m=1}^n\Stab_{\Sym(2^n)}(\mathcal{B}^n_m)=
 \Span{s_1,\ldots, s_n} \cong \wr_{i=1}^n C_2,\] which is the $n$-th
 iterated wreath product of copies of the cyclic group $C_2$ of order
 2, i.e.\ a Sylow $2$-subgroup of $\Sym(2^n)$.
 Notice that $\mathcal{B}^n_0 \succ \cdots \succ \mathcal{B}^n_n$ is
 an imprimitivity chain $\mathcal{C}_n$ of maximal depth for
 $\Sigma_n$ and that $\Sigma_n$ is the stabilizer of $\mathcal{C}_n$
 in $\Sym(2^n)$.\\

 Let $T_{n,0} \deq \Set{1}$ and, for $1 \leq i \leq n$, let us define
 $T_{n,i}\deq   \Span{t^n_1,\ldots,t^n_i}$.   Clearly
 $T_{n,i}\le T_{n,i+1}$, for $0 \leq i \leq n-1$. The group
 $T_n\deq T_{n,n}$ is a regular elementary abelian subgroup of
 $\Sym(2^n)$ of order $2^n$ contained in $\Sigma_n$, whose normalizer in $\Sym(2^n)$ is $\AGL(T_n)$, the affine general linear group. We also define
 \[
 U_n \deq \AGL(T_n)\cap \Sigma_n =N_{\Sigma_n}(T_n).
 \]
 The group $T_n$ is a uniserial module for $U_n$ whose \emph{maximal
 flag} $\mathcal{F}_n$ is defined as
 \begin{equation*}
 \Set{1}=T_{n,0} <\cdots< T_{n,n}=T_n.
 \end{equation*}
 
 \noindent Given a subgroup $H\leq \Sigma_{n-1}$, we define the \emph{diagonal
 embedding} of $H$ into $\Sigma_n$ as
 \[
 \Delta_n(H) \deq \{ (x,x^{s_n}) \mid x\in H\}.
 \]
 
 \begin{remark}\label{lem:one}
 It was already known to Dixon~\cite{Dixon1971} that the set of
 elementary abelian regular subgroups of $\Sym(2^n)$ form a unique
 conjugacy class. Moreover, a transitive abelian subgroup of
 $\Sym(2^n)$ is regular and so is self-celtralizing. In particular,
 $(T_n)^g$ is self-centralizing in $\Sigma_n$, for every
 $g\in \Sym(2^n)$.
 \end{remark}
 
 \begin{lemma}
 Up to conjugation by elements of $\Sigma_{n}$, the group $ T_n$ is
 the unique elementary abelian regular subgroup of $\Sym(2^n)$
 having $\mathcal{C}_n$ as imprimitivity chain.
 \end{lemma}

 \begin{proof}
 First, recall that $\Sigma_n$ stabilizes $\mathcal{C}_n$ for every
 $n$. We argue by induction on $n$, the result being trivial when
 $n=1$. Let $T$ be an elementary abelian regular subgroup of
 $\Sym(2^n)$ having $\mathcal{C}_n$ as imprimitivity chain and let
 $M$  be  the  stabilizer  in  $T$  of
 $\{1,\ldots,2^{n-1}\}=B^n_{1,0}\in \mathcal{B}_1^n$. In particular,
 $M$       stabilizes       also
 $B^n_{1,1}=(B^n_{1,0})^{s_n}=\{2^{n-1}+1,\ldots,2^{n}\} $. The
 group $M$ acts on $B^n_{1,0}$ as an elementary abelian regular
 subgroup $M_1$ of $ S_{2^{n-1}}$ having $\mathcal{C}_{n-1}$ as
 imprimitivity chain. By induction, $M_1=(T_{n-1})^{h_1}$ for some
 $h_1\in \Sigma_{n-1}$. Similarly, the group $M$ acts faithfully on $B^n_{1,1}$ as an
 elementary abelian regular subgroup $M_2$ of $(S_{2^{n-1}})^{s_n}$
 having $(\mathcal{C}_{n-1})^{s_{n}}$ as imprimitivity chain, and
 thus we find by induction $M_2=((T_{n-1})^{h_2})^{s_n}$ for some
 $h_2\in \Sigma_{n-1}$. Finally, we have that
 \[M=\Set{(m^{h_1},m^{    h_2s_n})\mid    m\in
  T_{n-1}}=\Delta_n(T_{n-1})^{(h_1,h_2^{s_n})}.\]
 If $t \in T\setminus M$ then $t$ interchanges $B^n_{1,0}$ and
 $B^n_{1,1}$ and centralizes $M$. Let us write $t$ in the form
 $t=(a,b^{s_n})s_n$,  where  $a,b\in  \Sigma_{n-1}$  and
 $(m^{h_1},m^{h_2s_n})=(m^{h_1},m^{h_2s_n})^t=(m^{h_2b},(m^{h_1a})^{s_n})$.
 Note that
 \begin{itemize}
 \item $1=t^2= (a,b^{s_n})s_n(a,b^{s_n})s_n=(ab,(ba)^{s_n})$, and so
   $a=b^{-1}$;
 \item $m^{h_1}= m^{h_2b}$ and $m^{h_2 s_n}=m^{h_1 a s_n}$ for all
  $m\in T_{n-1}$, from which we derive
  \[h_1ah_2^{-1},h_2bh_1^{-1}         \in
  C_{\Sigma_{n-1}}(T_{n-1})=T_{n-1},\]      i.e.\
  $a=h_1^{-1}uh_2=u^{h_1}h_1^{-1}h_2$       and
  $b=a^{-1}=h_2^{-1}h_1u^{h_1}= u^{h_2}h_2^{-1}h_1$ for some
  $u \in T_{n-1}$.
 \end{itemize}
Then we have
 \begin{align*}
  t&=(a,b^{s_n})s_n=( u^{h_1}h_1^{-1}h_2 , u^{h_2 s_n}(h_2^{-1}h_1)^{s_n} )s_n \\
  &\equiv ( h_1^{-1}h_2 , (h_2^{-1}h_1)^{s_n} )s_n = s_n^{(h_1,h_2^{s_n})}\bmod M.
 \end{align*}	 
 Since  $T_n=\Delta_n(T_{n-1})\rtimes  \Span{s_n}$,  then
 $T=T_n^{(h_1,h_2^{s_n})}$, as required.
 \end{proof}

 \begin{remark}\label{rmk_chainflag}
 Notice that the chain $\mathcal C_n$ is a maximal imprimitivity
 chain for $T_n$, even though it is not the only one. It is known that
 every maximal imprimitivity chain for $T_n$ determines and is
 determined   by   a   maximal   flag
 $\Set{1}=T_{n,0} <\cdots< T_{n,n}=T_n$. Indeed, the partition
 $\mathcal B_i$ is the set of the orbits of $T_{n,n-i}$, and
 conversely $T_{n,n-i}$ is the pointwise stabilizer of the action of
 $T_n$ over $\mathcal B_i$. Any Sylow 2-subgroup $U$ of
 $\AGL(T_n) $ is the stabilizer by conjugation
 of a maximal flag of $T_n$, and
 therefore it stabilizes also the associated imprimitivity chain. In
 particular, the stabilizer of $\mathcal C_n$ in $\AGL(T_n)$ is
 $U_n=\Sigma_n \cap \AGL(T_n)$. More generally, any maximal flag
 $\mathcal F$ of $T_n$ determines a Sylow 2-subgroup $U_\mathcal{F}$
 of $\AGL(T_n)$ and a Sylow 2-subgroup $\Sigma_\mathcal{F}$~\cite[Theorem p. 226]{Leinen1988} of
 $\Sym(2^n)$      such      that
 $U_\mathcal{F} = \Sigma_{\mathcal F} \cap \AGL(T_n)$. The maps
 $\mathcal  F  \mapsto  U_\mathcal  F$  and
 $\mathcal F \mapsto \Sigma_\mathcal F$ are injective.
 Consequently, for every Sylow 2-subgroup $U$ of $\AGL(T_n)$ there
 exists a unique Sylow 2-subgroup $\Sigma$ of $\Sym(2^n)$ such that
 $U=\Sigma \cap \AGL(T_n)$. In particular, the intersection
 $\AGL(T_n)\cap \Sigma_n=N_{\Sigma_n}(T_n)=U_n$ is a Sylow
 $2$-subgroup of $\AGL(T_{n})$.
 \end{remark}

 \section{Experimental evidence on a normalizer chain}\label{sec_two}
 Let us start by defining the
 normalizer chain of $T_n$.
 \begin{definition}
 Using the same notation of the previous
 section, 
 the normalizer chain of $T_n$ is defined as the sequence
 $\{N_n^k\}_{k \geq 0}$, where
 \[N^{0}_{n}\deq U_n = N_{\Sigma_n}(T_n), \quad N^{1}_{n} \deq
  N_{\Sym(2^n)}(U_n),\] and recursively, for $k > 1$,
 \[N^{k}_{n}\deq N_{\Sym(2^n)}(N^{k-1}_{n}).\]
 \end{definition}

 Considering $\Sigma_n$ in place of $\Sym(2^n)$ 
 the resulting
 chain is the same, as proven in the next theorem.
 \begin{theorem}\label{th:Nk(Nk-1)}
 For every $k\ge 1$, we have $N^{k}_{n}=N_{\Sigma_n}(N^{k-1}_{n})$. In particular, $N^{k}_{n}$ is a $2$-group.
 \end{theorem}
 \begin{proof}
 Suppose that $\mathcal{B}$ is a system of imprimitivity for
 $N^{k-1}_{n}$. For each $x\in N^{k}_{n}$, the partition
 $\mathcal{B}^x$ is a system of imprimitivity for $(N^{k-1}_{n})^x$
 and so for $N^{k-1}_{n}$, since $(N^{k-1}_{n})^x=N^{k-1}_{n}$.
 Thus, for a given $x\in N^{k}_{n}$ and an imprimitivity chain
 $\mathcal{C}$ for $N^{k-1}_{n}$, the set $\mathcal{C}^x$ is also an
 imprimitivity chain for $N^{k-1}_{n}$ and a fortiori for $U_n$.
 Since, by Remark~\ref{rmk_chainflag}, the imprimitivity chain
 $\mathcal{C}_n$ is the unique maximal one for $U_n=N^{0}_{n}$, we
 have $\mathcal{C}_n^x=\mathcal{C}_n$. Hence $\mathcal{C}_n$ is
 stabilized by $N^{k}_{n}$ for every $k$, yielding
 $N^{k}_{n}\le \Sigma_n$.
 \end{proof}
A direct consequence of the previous theorem is
 that there exists $d(n)\in  \mathbb N$ such that
 \[N^{k}_{n}=N^{d(n)}_{n} = \Sigma_n\]
  for every $k \geq d(n)$. We can
 interpret $d(n)+1$ as an upper bound for the \emph{defect}
 $\delta(n)$ of $T_n$ in $\Sigma_n$, i.e.\ the length of the shortest
 subnormal series from $T_n$ to $\Sigma_n$. Recalling that $\Sigma_n$ is self-normalizing in $\Sym(2^n)$~(see \cite{Carter1964}), as already pointed out in the introduction, 
 the fact previously stated represents a further
 argument showing that every Sylow $2$-subgroup of $\AGL(T_n)$ is
 contained in exactly one Sylow 2-subgroup of $\Sym(2^n)$.\\
 
We already recalled in Remark~\ref{rmk_chainflag} that $N^{0}_{n}=U_n$
normalizes a maximal flag $\mathcal{F}$ of $T_n$. Below we study the
action by conjugation of $N^{1}_{n}$ over $\mathcal{F}$.
\begin{proposition}
 The group $N^{1}_{n}$ normalizes each term of the flag
 $\{T_{n,0},\ldots, T_{n,n-2}\}$.
\end{proposition}
\begin{proof}
 It is enough to prove that each element of $N^{1}_{n}\setminus U_n$
 normalizes $T_{n,i}$ for $0\leq i \leq n-2$. For each
 $g\in N^{1}_{n}\setminus U_n$, from \cite[Corollary 3]{Aragona2019}
 we have that $T_{n,n-2}=T_n\cap T_n^g$ is normal in $N^{1}_{n}$.
 Hence, for every subgroup $H= T_{n,i}$ where $i < n-2$ and for every
 $g\in N^{1}_{n}\setminus U_n$, we have $H^g\le T_{n,n-2}$. If
 $x\in U_n$, we clearly have $(H^g)^x=(H^{gxg^{-1}})^g=H^g$. Since
 $T_n$ is a uniserial $U$-module, we conclude that $H^g$ belongs to
 $\mathcal{F}_n$. Thus $T_{n,i}^g=H^g=T_{n,i}$.
\end{proof}

We also used  \textsf{GAP}  to calculate
$N^{k}_{n}$ for $n\leq 11$.
The computational results are summarized in Fig.~\ref{table1}, where
the entry in position $(k, n)$ denotes the logarithm in base $2$ of
the size of $N^{k-1}_{n}$. We
observe that, in each column, consecutive values above the diagonal (bold values in the figure) have fixed differences. Such differences are listed in the (auxiliary) last column. 
For example, the number ``+7'' appearing in the last
position of the fifth row denotes that the difference between
$\log_{2} |N^{4}_{j}|$ and $\log_{2} |N^{3}_{j}|$ equals 7, where
$5 \leq j \leq 11$, reading the table from left to right, starting
from the position $(5,5)$ containing the bold number, i.e.\ the number
35.
 
\begin{figure}[htbp]
 \begin{center}
 \begin{tabular}{c||c|c|c|c|c|c|c|c|c|c||c}
  n & 2&3&4&5&6&7&8&9&10&11&\\
  \hline\hline
  $\log_{2} |U_n|$ & {\bf 3}&6&10&15&21&28&36&45&55&66 &\\
  \hline
  $\log_{2} |N^{1}_{n}|$ & -& {\bf 7} & 11 & 16 &22 & 29 & 37 & 46&56&67&+1\\
  \hline
  $\log_{2} |N^{2}_{n}|$ & -& -& {\bf 13} & 18 &24& 31 & 39 & 48&58&69&+2\\
  \hline
  $\log_{2} |N^{3}_{n}|$ & -& - & 14 & {\bf 22} &28&35 &43 & 52&62&73&+4\\
  \hline
  $\log_{2} |N^{4}_{n}|$ & -& - & 15 & 23 &{\bf 35}&42 &50 & 59&69&80&+7\\
  \hline
  $\log_{2} |N^{5}_{n}|$ & -& - & - & 25 &37&{\bf 53} & 61 & 70&80&91&+11\\
  \hline
  $\log_{2} |N^{6}_{n}|$ & -& - & - & 27 &41&57 & {\bf 77} & 86&96&107 &+16\\
  \hline
  $\log_{2} |N^{7}_{,n}|$ & -& - & - & 28 &45& 64 & 84 &{\bf 109}&119&130&+23\\
  \hline
  $\log_{2} |N^{8}_{n}|$ & -& - & - & 29 &46& 67 & 89 &113&{\bf 151}&162&+32\\
  \hline
  $\log_{2} |N^{9}_{n}|$ & -& - & - & 30 &47& 71 & 95 &122&155&{\bf 205}&+43\\
 \end{tabular}
 \end{center}
 \caption{The logarithm of the size of the normalizers, when
 $n \leq 11$}
 \label{table1}
\end{figure}

The table suggests that the values 
$\log_2\Size{N^{k}_{n}:N^{k-1}_{n}}$, reported in the last column of Fig.~\ref{table1}, do not
depend on $n$, if $n\geq k+2$, and match with
those of  the
sequence $\Set{a_j}_{j \geq 1}$ of the partial sums of the sequence $\Set{b_j}_{j \geq 1}$
counting the number of partitions of $j$ into at least two
distinct parts. The reader is referred to \emph{The On-Line Encyclopedia of
 Integer  Sequences} at~\cite[\url{https://oeis.org/A317910}]{OEIS} for a list of values and additional information.
In the next section we show that for small values of $k$ this is actually true. The above evidence is now summarized here as a conjecture.

\begin{conjecture}\label{conj:main}
 For $n\ge k+2\ge 3$, the number $\log_{2}\Size{N^{k}_n : N^{k-1}_n}$
 is independent of $n$ and is equal to $(k+2)$-th term of the
 sequence $\Set{a_j}_{j \geq 1}$ of the partial sums of the sequence $\Set{b_j}_{j \geq 1}$
 counting the number of partitions of $j$ into at least two
 distinct parts.
\end{conjecture}


\noindent The first values of the sequences $a_j$ and $b_j$ are listed in Fig.~\ref{tab:one}.
\\
\begin{figure}
 \centering
 \begin{tabular}{c||c|c|c|c|c|c|c|c|c|c|c|c|c|c}
 $j$  & $1$ & $2$ & $3$ & $4$ & $5$ & $6$ & $7$ & $8$ & 9& 10 &11&12&13&14\\
 \hline\hline
 ${b_j}$& 0& 0& 1& 1& 2& 3& 4& 5& 7& 9& 11& 14& 17& 21\\
 \hline
 $a_j$ &0& 0& 1& 2& 4& 7& 11& 16& 23& 32& 43& 57& 74&95
 \end{tabular}
 \caption{First values of the sequences $a_j$ and $b_j$}
 \label{tab:one}
\end{figure}


\section{Theoretical evidence}\label{sec:evi}
In this section we prove Theorem~\ref{thm:main} which 
solves
Conjecture~\ref{conj:main} in the cases $1\leq k\leq 4$, by providing
an explicit construction of $N^{k}_{n}$. We first need the following
general lemma.
\begin{lemma}\label{lem:two}
 Let $G=A\rtimes B$ be a group and $H$ a subgroup of $G$ containing
 $B$. If $[N_{A}(H\cap A), B]\le H$, then
 \[ N_G(H)=N_{A}(H\cap A)\rtimes B.
 \]
\end{lemma}

\begin{proof}
 Clearly  $B\le  H  \le  N_G(H)$.   Let
 $x\in N_{A}(H\cap A)\le A\trianglelefteq G$. Then $[H,x]\subseteq A$,
 since $A$ is normal in $G$. Let $h\in H$ and let us write $h=bk$ where $b\in B$
 and $k\in A\cap H$. We have $[h,x]=[bk,x]=[b,x]^k[k,x]\in H$, since
 $[b,x]\in H$ as $[N_{A}(H\cap A), B]\le H$ by hypotheses. Thus
 $[H,x]\subseteq H\cap A$ and thus $N_{A}(H\cap A)\le N_G(H)$.

 \noindent Let $x\in N_A(H\cap A)$, $k\in H\cap A$ and $b\in B$.
 Notice that
 \[k^{x^b}=((\underbrace{k^{b^{-1}}}_{\in H\cap A})^x)^b\in H\cap
 A.\] This implies that $N_{A}(H\cap A)$ is normalized by $B$. As
 a consequence, we have \[N_G(H)\ge N_{A}(H\cap A)\rtimes B.\]

 Conversely, let $x\in N_G(H)$. Since $G=A\rtimes B$, we can find
 $b\in B\le H\le N_G(H)$ such that $x=bu$ with $u\in A$. Clearly,
 $u\in N_G(H)\cap A=N_A(H)$. If $h\in H\cap A$, then
 $[u,h]\in H\cap A$, since $A$ is normal in $G$. Thus
 $x=bu\in N_{A}(H\cap A)\rtimes B$.
\end{proof}
 
Consider now the set-wise stabilizer $Q_{n}$ in $\Sigma_n$ of
$X_{n-1}$.   This  group  acts  also  on
$X_{n-1}^{s_n}=\{2^{n-1}+1,\ldots,2^n\}$    and   so
$Q_{n}=\Sigma_{n-1}\times   (\Sigma_{n-1})^{s_n}$   and
$\Sigma_n=Q_{n}\rtimes \Span{s_{n}}$, where $s_n$ interchanges the two
direct factors of $Q_n$. We can give a general procedure for the
construction of the normalizer $N_{\Sigma_n}(Y)$ of a subgroup
$Y\le  \Sigma_n$  containing  $T_n$  such  that
$[N_{Q_n}(Y\cap Q_n), s_n]\subseteq Y$. Since $t_n=s_n\in Y$, we have
$Y^{s_n}=Y$              and
$N_{\Sigma_n}(Y)=N_{Q_{n}}(Y\cap Q_{n})\rtimes\Span{s_n}$ by
Lemma~\ref{lem:two}.\\
 
Let us apply the previous construction to obtain a description of
$U_n$ as the normalizer of $T_n$ in $\Sigma_n$.
 
\begin{proposition}\label{prop_Un} 
 We have that
 \begin{align*}
 U_n & = \left< s_n\right> \ltimes \bigl( \Delta_n(U_{n-1}) \cdot ( T_{n-1}\times T_{n-1}^{s_n})\bigr)\\
  & = \left< s_n\right> \ltimes \bigl( \Delta_n(U_{n-1}) \cdot T_{n-1}\bigr).
 \end{align*}
\end{proposition}
\begin{proof}
 Using the same notation as above, we notice that
 \[
 T_n\cap Q_{n}=\Delta_n(T_{n-1})=\{ (t,t^{s_{n}}) \mid t\in
 T_{n-1}\}.\] We first claim that
 \[
 N_{Q_{n}}(T_{n}\cap Q_{n})= \Delta_n(U_{n-1}) \cdot T_{n-1}.
 \]
 It is straightforward to check that $\Delta_{n}(U_{n-1})$ normalizes
 $T_{n}\cap Q_{n}=\Delta_n(T_{n-1})$ and that $T_{n-1}$ centralizes
 $T_{n}\cap Q_{n}$, hence
 \[
 N_{Q_{n}}(T_{n}\cap Q_{n})\ge \Delta_n(U_{n-1}) \cdot T_{n-1}.
 \]
 Now, let $x=(a,b^{s_{n}})\in N_{Q_{n}}(T_{n}\cap Q_{n})$, where
 $a,b\in      \Sigma_{n-1}$,      and
 $y=(t,t^{s_n})\in \Delta_n(T_{n-1} ) =T_{n}\cap Q_{n} <
 T_{n-1}\times T_{n-1}^{s_n}$. We have that
 \[
 y^x= (t^a, (t^{b})^{s_n}) = (\bar{t},\bar{t}^{s_{n}})\in
 \Delta_n(T_{n-1} )<T_{n-1}\times T_{n-1}^{s_n}
 \]
 for some $\bar{t}\in T_{n-1}$,
 and so $t^a=\bar{t}=t^b$. It follows that $a,b\in U_{n-1}$ and
 $ab^{-1}\in   C_{\Sigma_{n-1}}(T_{n-1})=T_{n-1}$   by
 Remark~\ref{lem:one}.  Therefore $a=b  \tilde{t}$, with
 $\tilde{t}\in      T_{n-1}$      and
 $x=(b ,b^{s_{n}})\cdot(\tilde{t},1)\in \Delta_n(U_{n-1}) \cdot
 T_{n-1}$, giving the opposite inclusion. In conclusion,
 \[
 N_{Q_{n}}(T_{n}\cap Q_{n})= \Delta_n(U_{n-1}) \cdot T_{n-1}.
 \]
 In order to apply Lemma~\ref{lem:two}, it remains to be shown that
 $[\Delta_n(T_{n-1})\cdot T_{n-1}, s_n] \leq T_{n}$. Notice that
 $\Delta_n(T_{n-1})\cdot T_{n-1} = T_{n-1}\times T_{n-1}^{s_n}$, and
 thus
 $[\Delta_n(T_{n-1})\cdot T_{n-1}, s_n] \leq \Delta_n(T_{n-1}) \leq
 T_{n}$, as claimed.
\end{proof}
 
 \begin{proposition}\label{prop:normfact}
 Let  $H\unlhd  K  \le  \Sigma_{n-1}$  and
 $U \deq \Span{s_n} \ltimes \left( \Delta_n(K)\cdot (H\times
  H^{s_n}) \right)$. If we define
 \begin{itemize}
 \item $L\deq
  N_{\Sigma_{n-1}}(K)\cap N_{\Sigma_{n-1}}(H)$,
 \item $M\deq C_K(K/H)$,
 \end{itemize}
 then
 \[N_{\Sigma_n}(U)= \Span{s_n} \ltimes \left( \Delta_n(L)\cdot
  (M\times   M^{s_n})   \right).\]   Moreover,
 $M \unlhd L \le\Sigma_{n-1}$.
 \end{proposition}
 
 \begin{proof}
 The              inclusion
 $N_{\Sigma_n}(U) \ge \Span{s_n} \ltimes \left( \Delta_n(L)\cdot
  (M\times M^{s_n}) \right)$ is straightforward since every factor
 of the second member is contained in the first one.
 
 Note that $U\cap Q_n=\Delta_n(K)\cdot (H\times H^{s_n})$. Let us
 start considering the group
 \[N\deq N_{Q_n}(U) = N_{\Sigma_{n}}(U)\cap Q_n
  .\]   Let   $x=(a,b^{s_n})\in   N$,   where
 $a,b\in  \Sigma_{n-1}$.  
 First we note that $[x,s_n]=(a^{-1}b, (b^{-1}a)^{s_n})\in U\cap Q_n = \Delta_n(K)\cdot   (H\times
 H^{s_n})$. In particular, $a^{-1}b\in K$.
 
 Let
 $y=(h,1^{s_n})\in H\times H^{s_n}\le U\cap Q_n$, where  $h\in H$ and $1\in \Sigma_{n-1}$. 
 We have
 $y^x=(h^a,1^{s_n})\in   \Delta_n(K)\cdot   (H\times
 H^{s_n})=\Delta_n(K)\ltimes (H\times 1)$. Since
 $\Delta_n(K)\cap (H\times 1)=1$, then $h^a\in H$ for all
 $h\in H$ and so $a\in N_{\Sigma_{n-1}}(H)$.
 Similarly, we have that $b\in N_{\Sigma_{n-1}}(H)$.
 Now, letting $u=(k,k^{s_n})\in\Delta_n(K)$, we have
 \begin{eqnarray}\label{eq:norm_1}
  u^x=\left(k^a,(k^{b})^{s_n}\right)&=&\left(k^a,
           (k^a)^{s_n}\right)\cdot \left(1,((k^a)^{-1}k^b)^{s_n}\right) \\
          \notag & \in&
            \Delta_n(K)\cdot (H\times H^{s_n})\\
          \notag &=&\Delta_n(K)\ltimes (1\times
           H^{s_n}),
 \end{eqnarray}
 and so $a\in N_{\Sigma_{n-1}}(K)\cap N_{\Sigma_{n-1}}(H)=L$. Similarly, $b\in N_{\Sigma_{n-1}}(K)\cap N_{\Sigma_{n-1}}(H)=L$. Again by Eq.~\eqref{eq:norm_1}
 we have $b= am$ with $m=a^{-1}b\in C_L(K/H) \cap K=M$. It follows that \begin{equation*}
 x=(a,b^{s_n})=(a,a^{s_n})\cdot (1,m^{s_n}) \in \Delta_n(L)\cdot
 (M\times M^{s_n}). 
 \end{equation*}
 Hence $N \le \Delta_n(L)\cdot
 (M\times M^{s_n})$. As a consequence we have \[ N_{\Sigma_n}(U) =\Span{s_n}\ltimes N = \Span{s_n} \ltimes \bigl( \Delta_n(L)\cdot
 (M\times M^{s_n}) \bigr),\] as required.
 
  We also trivially have that $M\le L\le \Sigma_{n-1}$. If $m\in M$,
  $k\in K$ and $l\in L$ then 	
  \[
   [k,m^l]= l^{-1}\underbrace{[k^{l^{-1}},m]}_{\in H}l \in H, \] and
  therefore $m^l\in M$ and $M\unlhd L$.
 \end{proof}
 
 The following technical definition is necessary to provide a
 recursive construction for the normalizer chain of $T_n$ in
 $\Sigma_n$. 
 
\begin{definition}\label{def:CeD}
 For a given natural number $n$ we define the series
 $\Set{C^{k}_{n}}_{k\ge 0}$ and $\Set{D^{k}_{n}}_{k\ge 0}$
 recursively as follows:
 \begin{alignat*}{2}
 C^{0}_{n}&\deq T_n,&\ \ \\
 D^{0}_{n} &\deq N_{\Sigma_n}(C^{0}_{n})=U_n, \\
 C^{k}_{n} &\deq C_{D^{k-1}_{n}}\bigl(D^{k-1}_{n}/C^{k-1}_{n}\bigr) && \text{for $k\ge 1$},\\
 D^{k}_{n}  &\deq  N_{\Sigma_n}(C^{k-1}_{n})  \cap
 N_{\Sigma_n}(D^{k-1}_{n}) && \text{for $k\ge 1$}.
 \end{alignat*}
\end{definition}

\begin{proposition}\label{cor:Nkn}
 For each $k\ge 1$ we have that
 \begin{align*}
 N^{k}_{n} & =\Span{s_n} \ltimes \bigl(\Delta_n(D^{k}_{n-1}) \cdot \left(C^{k}_{n-1} \times (C^{k}_{n-1})^{s_n} \right) \bigr)\\
    & = \Span{s_n} \ltimes \bigl(\Delta_n(D^{k}_{n-1}) \ltimes \left(C^{k}_{n-1} \times \Set{1} \right) \bigr).
 \end{align*}
\end{proposition}
\begin{proof} The result follows by a recursive application of 
 Proposition~\ref{prop:normfact},  assuming  $H=C^{k-1}_{n-1}$,
 $K=D^{k-1}_{n-1}$, $L=D^{k}_{n-1}$ and $M=C^{k}_{n-1}$,
 beginning with
 $C^{0}_{n-1}=T_{n-1}$ which is normal in $D^{0}_{n-1}=U_{n-1}$. 
\end{proof}

\subsection{The case $1\leq k \leq 4$}\label{sec_sub}
The main result of this work will be proved in this section. In order
to do so, let us denote by
$\Theta_n$ the group of the upper unitriangular matrices and by
$Z_h(\Theta_n)$ the $h$-th term of its upper central series.

By Proposition~\ref{prop_Un}
\[
U_n=\Span{s_n} \ltimes \bigl(\Delta_n(U_{n-1}) \ltimes \left(T_{n-1}
\times \Set{1} \right)\bigr)
\]
and \[T_n=\Span{s_n} \cdot \Delta_n(T_{n-1}).\] Hence 
\[
\Theta_n \cong U_n/T_n=
\Delta_n(U_{n-1}/T_{n-1}) \ltimes \left(T_{n-1} \times \Set{1}
\right)
= \Delta_n(\Theta_{n-1}) \ltimes \left(T_{n-1} \times \Set{1}
\right).
\]
Moreover, notice that $\Theta_{n}=U_{n-1}$. It is easily checked that
\[
 Z_1(\Theta_n)= T_{n-1,1} \times \Set{1}.
\]
Proceeding by induction we obtain the following generalization.
\begin{lemma}\label{lem:upper_central}
 We have that
 \[
 Z_h(\Theta_n)=\Delta_n(Z_{h-1}(\Theta_{n-1}))    \ltimes
 \left(T_{n-1,h} \times \Set{1} \right).
 \]
\end{lemma}
\begin{proof}
 If
 $G_h\deq \Delta_n(Z_{h-1}(\Theta_{n-1})) \ltimes \left(T_{n-1,h}
 \times \Set{1} \right)$, then $G_h/G_{h-1}$ is a central section
 of $\Theta_n$, hence $G_{h}\leq Z_{h}(\Theta_n)$. Notice that
 $\Size{G_h:  G_{h-1}}=\Size{Z_{h}(\Theta_n):Z_{h-1}(\Theta_n)}$,
 which is known to be $2^h$. Therefore $Z_{h}(\Theta_n)=G_h$.
\end{proof}
We are now ready to prove our main result.
\begin{theorem}\label{thm:main}
 Let $n$ be a non-negative integer. Then Conjecture~\ref{conj:main}
 is true for $1 \leq k \leq 4$.
\end{theorem}
\begin{proof}
 Let us prove each case separately. We will use Proposition~\ref{prop:normfact} repeatedly without further mention.
 Since $C^{0}_{n}=T_n$ and $D^{0}_{n}=U_n$, by
 Proposition \ref{prop_Un} we have
 \[
 N^{0}_{n} =\Span{s_n} \ltimes \bigl(\Delta_n(U_{n-1}) \ltimes
 \left(T_{n-1} \times \Set{1} \right)\bigr)=U_n.
 \]
 \begin{enumerate}
 \item[$\mathbf{[k=1]}$]          Since
 $C^{1}_{n}=C_{U_n}(U_n/T_n)=Z_1(\Theta_n)\ltimes T_n $  and
 $D^{1}_{n}=U_n \cap N_{\Sigma_n}(U_n)=U_n$, we obtain
$$
N^{1}_{n} =\Span{s_n} \ltimes \bigl(\Delta_n(U_{n-1}) \ltimes
\left((Z_1(\Theta_{n-1})\ltimes T_{n-1}) \times \Set{1} \right)\bigr),
$$
and so $\Size{N^{1}_{n}:N^{0}_{n}}=\Size{N^{1}_{n}:U_{n}}=2=2^1$, since $\Size{Z_1(\Theta_{n-1})}=2$.
\item[$\mathbf{[k=2]}$]      We     have
 $C^{2}_{n}=C_{U_n}\bigl(U_n/\left(Z_1(\Theta_n)\ltimes T_n
 \right)\bigr)=Z_2(\Theta_n)\ltimes T_n$ and
 \begin{align*}
 D^{2}_{n} & =N_{\Sigma_n}(C^{1}_{n}) \cap N_{\Sigma_n}(U_n)\\
    & =N_{\Sigma_n}(Z_1(\Theta_n)\ltimes T_n) \cap N^{1}_{n} \\
    & =N_{N^{1}_{n}}(Z_1(\Theta_n)\ltimes T_n)\\
    &=U_n.
 \end{align*}
 The last equality depends on the fact that
 \[
 \Size{T_n \cdot T_n^g} = 2^{n+2},
 \] where $g \in N^{1}_{n} \setminus U_n$ from \cite[Corollary
 3]{Aragona2019},      and      that
 $\Size{(Z_1(\Theta_n)\ltimes T_n)} = 2^{n+1}$. We consequently
 obtain that
 \[
 N^{2}_{n} =\Span{s_n} \ltimes \bigl(\Delta_n(U_{n-1}) \ltimes
 \left(( Z_2(\Theta_{n-1})\ltimes T_{n-1} ) \times  \Set{1}
 \right)\bigr),
 \]
 and so $\Size{N^{2}_{n}:N^{1}_{n}}=2^2=4$,
 since $\Size{Z_2(\Theta_{n-1}):Z_1(\Theta_{n-1})}=2^2$.
\item[$\mathbf{[k=3]}$]   We    have   that
 $C^{3}_{n}=C_{U_n}(U_n/\bigl( Z_2(\Theta_n) \ltimes T_n
 )\bigr)= Z_3(\Theta_n) \ltimes T_n$ and
 \begin{align*}
 D^{3}_{n} & =N_{\Sigma_n}(Z_2(\Theta_n) \ltimes T_n) \cap N_{\Sigma_n}(U_n)\\
    & =N_{\Sigma_n}(Z_2(\Theta_n) \ltimes T_n) \cap N^{1}_{n}\\
    & =N_{N^{1}_{n}}(Z_2(\Theta_n) \ltimes T_n) \\
    & =N^{1}_{n}.
 \end{align*}
 In order to prove last equality, we first show that
 $[N^{1}_{n},T_n] \leq Z_2(\Theta_n) \ltimes T_n $. For each
 $g \in N^{1}_{n} \setminus U_n$, we have that $[[T_n^g,U_n],U_n]$ is
 a normal subgroup of $U_n$ of index at least $4$ in $T_n^g$. Since $T_n^g$ is
 uniserial for $U_n$, then $[[T_n^g,U_n],U_n]$ is necessarily
 contained in the unique subgroup of index $4$ in $T_n^g$ and normal in
 $U_n$, which is $T_n^g\cap T_n$ (see \cite{Aragona2019}). Hence
 $(T_n\cdot T_n^g)/T_n$ lies in the second term of the upper central series of the quotient $U_n/T_n=\Theta_n$. Thus 
 $T_n^g\leq Z_2(\Theta_n) \ltimes T_n$.
 We   are   left  with   proving   that
 $[N^{1}_{n},Z_2(\Theta_n)] \leq Z_2(\Theta_n) \ltimes T_n$, which is
 a direct consequence of the following straightforward properties:
 \begin{itemize}
 	\item $Z_2(\Theta_n)=\Delta_n(Z_{1}(\Theta_{n-1}))    \ltimes
 	\left(T_{n-1,2} \times \Set{1} \right)$ (Lemma~\ref{lem:upper_central});
 \item
 $[s_n, \Delta_n(Z_1(\Theta_{n-1})) \ltimes \left(T_{n-1,2} \times
  \Set{1} \right)] \leq T_n$;
 \item
 $[\Delta_n(U_{n-1}), \Delta_n(Z_1(\Theta_{n-1})) ] \leq T_n$;
 \item
 $[\Delta_n(U_{n-1}), T_{n-1,1} \times \Set{1} ]\leq T_{n-1,2}
 \times \Set{1} \leq Z_2(\Theta_n)$;
 \item
 $[ Z_1(\Theta_{n-1} )\ltimes T_{n-1}, \Delta_n(Z_1(\Theta_{n-1}))
 ]\leq T_{n-1,1} \times \Set{1} = Z_1(\Theta_n)\leq Z_2(\Theta_n)$;
 \item 
 $[Z_1(\Theta_{n-1}) \ltimes T_{n-1}, T_{n-1,2} \times \Set{1} ]=
 \Set{1}$.
 \end{itemize}
 In conclusion, we derive that
$$
N^{3}_{n} =\Span{s_n} \ltimes \bigl(\Delta_n(N^{1}_{n-1}) \ltimes
\left(( Z_3(\Theta_{n-1}) \ltimes T_{n-1} ) \times \Set{1} \right)\bigr),
$$
and so $\Size{N^{3}_{n}:N^{2}_{n}}=2^4=16$,
as $\Size{N^{1}_{n-1}:U_{n-1}}=2^1$ and
$\Size{Z_3(\Theta_{n-1}):Z_2(\Theta_{n-1})}=2^3$.
The same result can be also obtained as follows. 
Note that $Z_2(\Theta_n)\ltimes T_n =Z_{3}(U_{n})\cdot T_n$. Since $Z_{3}(U_{n})$ is a characteristic subgroup of $U_n$ we have  $(Z_2(\Theta_n)\ltimes T_n)^x =(Z_{3}(U_{n})\cdot T_n)^x= Z_{3}(U_{n})\cdot T_n^x \le Z_2(\Theta_n)\cdot T_n\cdot T_n^x= Z_2(\Theta_n)\ltimes T_n$ for all $x\in N_n^1$.
As a consequence $[N^{1}_{n},Z_2(\Theta_n)] \le [N^{1}_{n},Z_2(\Theta_n)T_n] \leq Z_2(\Theta_n) \ltimes T_n$.
\item[$\mathbf{[k=4]}$]
We mimic the argument provided for the case $k=3$.
We start by computing
 \begin{align*}
 C^{4}_{n} & =C_{N^{1}_{n}}\bigl(N^{1}_{n}/(Z_3(\Theta_n)\ltimes T_n)\bigr)\\
    & = C_{M}\left(\frac{\Delta_n(\Theta_{n-1})\ltimes \bigl((Z_1(\Theta_{n-1}) \ltimes T_{n-1}) \times \Set{1} \bigr)}{\Delta_n(Z_2(\Theta_{n-1}))\ltimes \bigl((Z_1(\Theta_{n-1}) \ltimes T_{n-1,3}) \times \Set{1}\bigr)}\right)\ltimes T_n\\
    &= \bigl(\Delta_n(Z_3(\Theta_{n-1}))\ltimes \bigl((Z_1(\Theta_{n-1}) \ltimes T_{n-1,4}) \times \Set{1}\bigr)\bigr)\ltimes T_n,
 \end{align*}
 where
 $M\deq  \Delta_n(\Theta_{n-1})\ltimes  \bigl((Z_1(\Theta_{n-1}) \ltimes T_{n-1}) \times \Set{1} \bigr)$. Moreover
 \begin{align*}
 D^{4}_{n} & =N_{\Sigma_n}(Z_3(\Theta_n)\ltimes T_n ) \cap N_{\Sigma_n}(N^{1}_{n})\\
    & =N_{\Sigma_n}(Z_3(\Theta_n)\ltimes T_n ) \cap N^{2}_{n}\\
    & =N_{N^{2}_{n}}(Z_3(\Theta_n)\ltimes T_n )\\
    & =N^{2}_{n}.
 \end{align*}
The last equality is derived proceeding as in the case $k=3$, provided
 that $n$ is sufficiently large (e.g.~$n\geq 5$). In conclusion, we
 derive that
 \[
 N^{4}_{n} =\Span{s_n} \ltimes \bigl(\Delta_n(N^{2}_{n-1}) \ltimes
 \left(C^{4}_{n-1} \times \Set{1} \right)\bigr),
 \]
 and  so  $\Size{N^{4}_{n}:N^{3}_{n}}=2^7$. Indeed, $\Size{N^{2}_{n-1}:N^{1}_{n-1}}=2^2$ and
 $\Size{C^{4}_{n-1}:C^3_{n-1}}=2^5$, since
 \[
 C^{4}_{n-1}=\bigl(\Delta_{n-1}(Z_3(\Theta_{n-2}))\ltimes
 \bigl((Z_1(\Theta_{n-2})\ltimes  T_{n-2,4} )   \times
 \Set{1}\bigr)\bigr)\ltimes T_{n-1}
 \]
 and
 \[
 C^{3}_{n-1}=\bigl(\Delta_{n-1}(Z_{2}(\Theta_{n-2}))  \ltimes
 \left(T_{n-2,3} \times \Set{1} \right)\bigl)\ltimes T_{n-1},
 \]
 where    $\Size{Z_3(\Theta_{n-2}):Z_2(\Theta_{n-2})}=2^3$,
 $\Size{T_{n-2,4},T_{n-2,3}}=2$ and $\Size{Z_1(\Theta_{n-2})}=2$.
\end{enumerate}
\bigskip

Finally, notice that, if $n\ge k+2$, the construction of $N^{k}_{n}$ described above
does not depend on the dimension $n$ and the integers corresponding to
$\log_2\Size{N^{k}_{n}:N^{k-1}_{n}}$ for $1 \leq k \leq 4$ are
respectively the 3-rd, the 4-th, the 5-th and the 6-th term of the
sequence $\Set{a_j}$ of the partial sums of the sequence $\Set{b_j}$
counting the number of partitions of $j$ into at least two distinct
parts (see Fig.~\ref{tab:one} and the auxiliary column of Fig.~\ref{table1}). Therefore, Conjecture~\ref{conj:main}
is true for $1 \leq k \leq 4$.
\end{proof}
Although we have a recursive method, it appears that the construction
of $N^{k}_{n}$ requires \emph{ad hoc} computations that become
increasingly complex as $k$ grows. 

\begin{openprob}
 Find a closed and concise formula for $N^{k}_{n}$.
\end{openprob}

Moreover, even though the sequence of the indices
$\log_2\Size{N_n^{k}:N_n^{k-1}}$ seems to be predictable for $k \leq n-2$,
as conjectured in this paper, it is hard to figure any conjecture on the values appearing under the bold diagonal of the table in Fig.~\ref{table1}.

\begin{openprob}
 Determine $\log_2\Size{N_n^{k}:N_n^{k-1}}$ for all natural numbers $k$ and $n$.
\end{openprob}

\section{Some computational aspects}\label{sec_construction}
We can derive from Proposition~\ref{prop_Un}  the following efficient
construction of $U_n$, which has been useful to speed up the process
of computing the normalizer chain.
\\

The center $Z(U_n)$ is the subgroup $\Span{t^n_1}$, which is actually
the center of $\Sigma_n$. Let
\[
 u^n_{1,j} \deq t^{n}_{n-j+1} \text{ for $1\le j\le n$}, \quad
 u^n_{2,j} \deq u^{n-1}_{1,j-1} \text{ for $2\le j\le n$},\] and for
$3\le i\le j\le n$
\[ u^n_{i,j} \deq u^{n-1}_{i-1,j-1}(u^{n-1}_{i-1,j-1})^{s_n}.\]

Using this notation, it easy to recognize that
\begin{align*}
 U_1&=T_1,\\
 U_2&=\Span{u^2_{1,1},u^2_{1,2},u^2_{2,2}},\\
 U_3&=\Span{u^3_{1,1},u^3_{1,2},u^3_{1,3}, u^3_{2,2}, u^3_{2,3}, u^3_{3,3}},\\
  &\ \ \vdots \\
 U_n&=\Span{u^n_{i,j} \mid 1\le i \le j \le n}.
\end{align*}
\medskip

As an example of this construction, we conclude the paper by showing
the \textsf{GAP} code which we used to build the normalizer chains
displayed in Fig.~\ref{table1}. The orders of the normalizer are also
provided. The code below is specialized to the case $n=8$.\bigskip

\begin{lstlisting}[language=GAP]
dim:=8;

gens:=[]; 
# will contain generators for T_n 
 
ngens:=[]; 
# will contain generators for U_n 
 
sgens:=[]; 
# will contain generators for Sigma_n

# construction of the previous list 
for i in [1..dim] do 
	x:=(); 
	for j in [1..2^(i-1)] do 
		x:=x*(j,j+2^(i-1)); 
	od; 
	newgens:=[];
	newngens:=[]; 
	for y in gens do 
		Add(newgens, y*y^x); 
	od; 
	for y in ngens do
		Add(newngens, y*y^x); 
		Append(newngens,gens); 
	od;
	newgens:=Set(newgens); 
	newngens:=Set(newngens); 
	gens:=newgens;
	ngens:=newngens; 
	Add(gens, x); 
	Add(ngens,x); 
	Add(sgens,x);
	tmpsyl:=Group(ngens);   
	ngens:=MinimalGeneratingSet(tmpsyl);
	tmpsyl:=false; 
od;
 
t:=Group(gens); # the group T_n 
sigma:=Group(sgens); # the group
Sigma_n u:= Group(ngens); # the group U_n
 
sym:=SymmetricGroup(2^dim); 
n:=u;

# here the normalizer chain computation starts
sz:=Collected(Factors(Size(n))); 
# the orders of the normalizer chain as power of 2
       
lst:=[[t,Collected(Factors(Size(t)))],[n,sz]]; 
# will contain the normalizers and their orders

flag:=true; 
Print(Collected(Factors(Size(u))));

# construction of the normalizers and order display 
while flag do
	m:=n;  
	n:=Normalizer(sym,m);  
	if n<>m then
		sz:=Collected(Factors(Size(n))); 
		Print(sz,"\n"); Add(lst,[n,sz]);
	else 
		flag:=false; 
	fi; 
od;
\end{lstlisting}

\section*{Acknowledgment}
We are thankful to the staff of the Department of Information Engineering, Computer Science and Mathematics at the University of L'Aquila for helping us in managing the HPC cluster CALIBAN, which we extensively used to run our simulations (\url{caliband.disim.univaq.it}). We are also grateful to the \emph{Istituto Nazionale d'Alta Matematica - F.\ Severi} for regularly hosting our research seminar \emph{Gruppi al Centro} in which this paper was conceived.

\bibliographystyle{alpha} \bibliography{sym2n_ref}

\end{document}